\documentclass[12pt]{amsart}
\usepackage{amssymb}
\usepackage[all]{xy}
\usepackage{shadow}
\usepackage{times}
\usepackage{enumerate}
\usepackage{mathrsfs}

\numberwithin{equation}{section}

\newtheorem{theorem}{Theorem}[section]
\newtheorem{proposition}[theorem]{Proposition}
\newtheorem{lemma}[theorem]{Lemma}
\newtheorem{corollary}[theorem]{Corollary}

\theoremstyle{definition}
\newtheorem{definition}[theorem]{Definition}
\newtheorem{example}[theorem]{Example}

\theoremstyle{remark}
\newtheorem{remark}[theorem]{Remark}

\newcommand{\ShortExactSeq}[3]%
{0\to{#1}\to{#2}\to{#3}\to 0}
\newcommand{\DShortExactSeq}[3]%
{\xymatrix@1{0\ar[r]&{{#1}}\ar[r]&{{#2}}\ar[r]&{{#3}}\ar[r]&0}}

\newcommand{\Ad}{\operatorname{\rm Ad}}

\newcommand{\Bounded}{{\frak B}}

\newcommand{\Compacts}{{\frak K}}

\newcommand{\Dlim}{\mbox{$\lim\limits_\rightarrow$}}

\newcommand{\R}{\mbox{$\Bbb R$}}

\newcommand{\Supp}{\operatorname{\rm Supp}}

\renewcommand{\le}{\leqslant}
 \renewcommand{\phi}{\varphi}
\renewcommand{\epsilon}{\varepsilon}

\newcounter{ritmctr}
{\end{itemize}}
\newcounter{aitmctr}
{\end{itemize}}

\begin{document}
\bibliographystyle{plain}

\title{Sheaf theory and Paschke duality}

\author{John Roe}
\address{Department of Mathematics, Penn State University, University
Park PA 16802}
\author{Paul Siegel}
\address{Department of Mathematics, Columbia University, 2990 Broadway, New York, NY 10027}
\date{\today}
\maketitle

\section{Introduction}
Let $X$ be a locally compact, metrizable space and let $H$ be an \emph{$X$-module} --- that is, a Hilbert space equipped with a representation $\rho$ of $C_0(X)$.  An operator $T\in\Bounded(H)$ is called \emph{pseudolocal} if $T\rho(f)-\rho(f)T$ is compact for all $f\in C_0(X)$, and it is called \emph{locally compact} if $T\rho(f)$ and $\rho(f)T$ individually are compact for all $f\in C_0(X)$.   Plainly, the pseudolocal operators form a $C^*$-algebra, and the locally compact operators form an ideal in this $C^*$-algebra.  We will denote these by ${\mathfrak D}(X)$ and ${\mathfrak C}(X)$ respectively, and we will denote by ${\mathfrak Q}(X)$ the quotient ${\mathfrak D}(X)/{\mathfrak C}(X)$.

Since the fundamental works of Atiyah~\cite{At7}, Brown-Douglas-Fillmore~\cite{BDF2} and Kasparov~\cite{Kas1}, it has been clear that the algebra extension
\[ 0 \to {\mathfrak C}(X) \to {\mathfrak D}(X) \to {\mathfrak Q}(X) \to 0 \]
functions as an abstract counterpart to the extension associated to the zero'th order pseudodifferential operators on a compact manifold $M$,
\[ 0 \to \Compacts \to \Psi(M) \to C(S^*M) \to 0. \]
Here $\Psi(M)$ is the $C^*$-algebra generated by the zero'th order pseudodifferential operators acting on $L^2(M)$, $S^*M$ is the unit cosphere bundle of $M$, and $\sigma\colon \Psi(M) \to C(S^*M)$ is the symbol map --- the map that takes a pseudodifferential operator to its principal symbol.

The algebra ${\mathfrak Q}(X)$ is not commutative, but it shares some of the properties of the symbol algebra $C(S^*M)$.  In particular, it is localizable: given an open subset $U$ of $X$, it makes sense to restrict an element of ${\mathfrak Q}(X)$ to $U$, and it also makes sense to glue together such restrictions on overlapping open sets.. Moreover, these considerations have been important in the study of index theory and K-homology.  The correct language in which to discuss such localization and gluing is, of course, that of sheaf theory~\cite{godement_topologie_1998}.  The purpose of this note is to study ${\mathfrak Q}(X)$ systematically from a sheaf-theoretic perspective.

\section{Sheaves of $C^*$-algebras}

Let $X$ be a locally compact and metrizable space.  The standard notations $C_0(X)$ and $C_b(X)$ will be used for the $C^*$-algebras of continuous functions on $X$ which vanish at infinity (respectively, are uniformly bounded).

In this section we will review some basic facts about sheaves of $C^*$-algebras.  Most if not all of this material is in the literature already (see~\cite{ara_sheaves_2010} and the other references in the appendix to this section) but is included here for completeness.

\begin{definition} A \emph{presheaf of $C^*$-algebras} over $X$ is a contravariant functor from the category of open subsets of $X$ and inclusions to the category of $C^*$-algebras.  Similarly we may speak of a presheaf of \emph{unital} $C^*$-algebras. \end{definition}

\begin{example}\label{psbf} The assignment $U \mapsto C_b(U)$ defines a presheaf of unital $C^*$-algebras. \end{example}

As usual, a \emph{sheaf} of $C^*$-algebras is a presheaf which ``can be reconstructed from local data''.   However, this notion  needs to be correctly interpreted to incorporate the analysis---this is not a sheaf of algebras in the classical sense of~\cite{godement_topologie_1998}.  Rather, it is a sheaf \emph{in the category of $C^*$-algebras.} We adopt the following  definition from~\cite{ara_sheaves_2010}.

\begin{definition}\label{ssdef} A \emph{sheaf of $C^*$-algebras} over  $X$ is a presheaf $\mathfrak A$ of $C^*$-algebras such that ${\mathfrak A}(\emptyset) = 0$ and such that, whenever $U = \bigcup U_j$ is a union of open subsets of $X$,
 the following additional axioms are satisfied:
\begin{enumerate}[(i)]
\item (Uniqueness) If $a\in {\mathfrak A}(U)$ and $a_{|U_j} = 0$ for all $j$, then $a=0$.
\item (Gluing) If $a_j \in {\mathfrak A}(U_j)$ is a \emph{bounded} family for which $a_{j|U_i} = a_{i|U_j}$ for all $i,j$, then there exists $a\in{\mathfrak A}(U)$ such that $a_{U_j} = a_j$ for all $j$.  Moreover, $\|a\| \le \sup_j \|a_j\|$.
\end{enumerate}
(This differs from the classical algebraic notion of sheaf by the \emph{boundedness} requirement appearing in the gluing axiom.) \end{definition}

The axioms can be expressed more categorically by saying that, whenever $U=\bigcup U_j$, the algebra ${\mathfrak A}(U)$ is the inverse  limit (in the category of $C^*$-algebras) of the diagrams ${\mathfrak A}(U_i) \to {\mathfrak A}(U_i\cap U_j)$, as $i$ and $j$ vary.  Beware that an ``inverse limit in the category of $C^*$-algebras'' is \emph{not} the same as an ``inverse limit of $C^*$-algebras'' or ``pro-$C^*$-algebra'' in the sense of Phillips~\cite{Phil88}; these latter objects are inverse limits of $C^*$-algebras \emph{in the category of topological algebras}.

\begin{example} It is easy to see that the presheaf of bounded continuous functions on $X$ (Example~\ref{psbf}) is in fact a sheaf. \end{example}

\begin{definition} We will say that a sheaf $\mathfrak A$ of unital $C^*$-algebras is a \emph{special} sheaf if it is a sheaf of (unital) modules over the sheaf of bounded continuous functions --- that is to say, if ${\mathfrak A}(U)$ is a module over $C_b(U)$ for every open set $U$, with the constant function 1 acting as the identity, and the natural compatibility relations hold. \end{definition}

All the examples that we shall consider in this paper will be special sheaves. In fact, it is apparently unknown whether there exist any non-special sheaves of $C^*$-algebras (see~\cite[Section 5]{ara_sheaves_2010}, where our ``special'' sheaves are called \emph{$\mathfrak C$-sheaves}).  Special sheaves can be identified with several other notions of ``continuous family of $C^*$-algebras'' over $X$; we discuss this briefly in the appendix to this section, but we won't make use of those results in the main text.

\begin{lemma}\label{extl} Suppose that ${\mathfrak A}$ is a special sheaf over $X$.  Let $U,V$ be open subsets of $X$ with $\bar{V}\subseteq U$.  Then for any $a\in{\mathfrak A}(U)$ there exists $a'\in{\mathfrak A}(X)$ with $a_{|V} = a'_{|V}$ and $\|a'\|\le \|a\|$. \end{lemma}

\begin{proof} Choose $g\in C_b(U)$ with $0\le g\le 1$, $\Supp(g)\subseteq U$ and $g=1$ on $\bar{V}$.  Let $W$ be the complement of $\Supp(g)$.  Apply the gluing axiom to the elements $g a \in {\mathfrak A}(U)$ and $0\in{\mathfrak A}(W)$.  \end{proof}

The \emph{stalk} of a sheaf (or presheaf) $\mathfrak A$ at $x\in X$ is defined as usual by
\[ {\mathfrak A}(x) = \Dlim \{ {\mathfrak A}(U): \mbox{$U$ open, $U\ni x$} \}, \]
the direct limit ($=$ colimit) of course being in the category of $C^*$-algebras.  For special sheaves, this is equivalent to a more concrete definition.  Suppose that $\mathfrak A$ is special, let $U$ be any open set containing $x$, and let $A = {\mathfrak A}(U)$.  Let $I\lhd C_b(U)$ be the ideal of functions vanishing at $x$.  By the Cohen-Hewitt factorization theorem, $IA$ is a closed ideal in $A$.  If $a\in  IA$, then for each $\epsilon>0$ there is an open $U\ni x$ such that $\|a_{|U}\|<\epsilon$, and so (by definition of the $C^*$-algebraic direct limit) the natural $*$-homomorphism $A \to {\mathfrak A}(x)$ vanishes on $IA$.  Thus  we obtain a $*$-homomorphism
\begin{equation}\label{stalk}\alpha\colon  A/IA \to {\mathfrak A}(x). \end{equation}
\begin{lemma} For any special sheaf $\mathfrak A$, the $*$-homomorphism $\alpha$ described in Equation~\ref{stalk} is an \emph{isomorphism}. \end{lemma}
\begin{proof} We must show that the map is both surjective and injective.

For surjectivity, consider the $*$-homomorphisms ${\mathfrak A}(U) \to {\mathfrak A}(x)$, as $U$ runs over a system of neighborhoods of $x$.  By definition of the direct limit, the union of the ranges of these homomorphisms is dense in ${\mathfrak A}(x)$.  But it follows from Lemma~\ref{extl} that this union is included in the range of $\alpha$.  Therefore $\alpha$ is surjective, since the image of a $*$-homomorphism is always closed.

For injectivity, suppose that $a\in {\mathfrak A}(X)$ maps to zero in ${\mathfrak A}(x)$. Then for each $\epsilon>0$ there is a neighborhood $U$ of $x$ such that $\|a_{|U}\|<\epsilon$.  Using the special structure, one easily produces $a'=\phi a$, where $\phi$ vanishes near $x$ and is 1 outside $U$, such that $\|a-a'\|<\epsilon$.  But clearly $a'\in IA$.  Since $IA$ is closed, the result follows.
  \end{proof}

As expected, an element of a sheaf is determined by its values on the stalks (for this reason we may refer to its as a ``section'' of the sheaf).

\begin{proposition}\label{locsec} Let $\mathfrak A$ be a sheaf of $C^*$-algebras over $X$ and let $a\in{\mathfrak A}(X)$.  For each $x\in X$ let $a(x)\in{\mathfrak A}(x)$ denote the value of $a$ at $x$.  If $a(x)=0$ for all $x\in X$, then $a=0$. \end{proposition}
\begin{proof} Let $\epsilon>0$.  The hypothesis implies that each $x\in X$ has a neighborhood $U_x$ such that $\|a_{|U_x}\|<\epsilon$.  By the gluing axiom, $\|a\|\le \epsilon$.  Since $\epsilon$ is arbitrary, $a=0$. \end{proof}

\begin{remark} It follows from this proposition that
the obvious $*$-homomorphism
\[ {\mathfrak A}(X) \to  \prod_{x\in X} {\mathfrak A}(x) \]
is injective.
Since an injective $*$-homomorphism is isometric, we find that for any section $a\in{\mathfrak A}(X)$,
\[ \|a\| = \sup \{ \|a(x)\|: x\in X\}. \]
We will need this identity in a moment.  Notice in particular the following consequence (called \emph{local convexity}): if $\mathfrak A$ is a \emph{special} sheaf, $\{\phi_j\}$ is a finite partition of unity, and $a_j \in {\mathfrak A}(X)$, then
\[ \left\| \sum \phi_j a_j \right\| \le \max \left\{ \|a_j\| \right\}. \]
The local convexity property was introduced in~\cite{hofmann_bundles_1977}.
 \end{remark}

We can glue local sections of special sheaves using a partition of unity.

\begin{lemma}\label{puglue} Let ${\mathfrak A}$ be a special sheaf over $X$.  Let ${\mathscr U} = \{U_j\}$ be an open cover of $X$ and let $\{\phi_j\}$ be a locally finite continuous partition of unity subordinate to $\mathscr U$.  Suppose that $a_j\in \mathfrak{A}(U_j)$ is a uniformly bounded family.  Then the sum
\[ a = \sum \phi_j a_j \]
defines an element of ${\mathfrak A}(X)$ with $\|a\| \le \sup_j \|a_j\|$.  Moreover, if $\|a_{j|U_i} - a_{i|U_j}\|<\epsilon$ for all $i,j$, then $\|a_{|U_j} - a_j\|\le \epsilon$ for all $j$. \end{lemma}

\begin{proof} Let $M= \sup \|a_j\|$.  Each point of $X$ has a neighborhood $W$ that meets the support of only finitely many $\{\phi_j\}$.   Then the sum
\[ a_W = \sum_j \phi_{j|W} a_{j|W} \]
is a finite one and defines $a_W \in {\mathfrak A}(W)$; moreover, by local convexity, $\|a_W\| \le M$.  Let $\mathscr W$ be an open cover of $X$ by sets $W$ arising as above, and observe that if $W,W'\in\mathscr W$ then $a_{W|W'} = a_{W'|W}$.  Thus, by the gluing axiom, there is $a\in{\mathscr A}(X)$ restricting to $a_W$ on each $W$, and $\|a\|\le M$.  This is the interpretation of the sum appearing in the statement of the lemma.

It remains to prove the final sentence.  However, if $\|a_{j|U_i} - a_{i|U_j}\|<\epsilon$ for all $i,j$, then we may write
\[ a_{|U_j} - a_j = \sum_i \phi_i (a_i-a_j)_{|U_j} \]
which has norm $\le\epsilon$ by local convexity again.
 \end{proof}

\emph{Morphisms} of sheaves are defined in the usual way (as natural transformations of the underlying presheaves).  A morphism of sheaves gives rise to a morphism on each stalk.  In the presence of appropriate local continuity the converse is also true:

\begin{proposition} Let $\mathfrak A$ and ${\mathfrak A}'$ be sheaves of $C^*$-algebras over $X$, and suppose that for each $x$ there is given a $*$-homomorphism $\alpha(x)\colon {\mathfrak A}(x) \to {\mathfrak A}'(x)$.  Suppose also that there is an open cover $\mathscr U$ of $X$ such that for each $U\in\mathscr U$ there is a $*$-homomorphism $\alpha(U)\colon {\mathfrak A}(U)\to {\mathfrak A}'(U)$ such that $\alpha(x)$ is the germ of $\alpha(U)$ for all $x\in U$.  Then there is one and only one morphism of sheaves $\alpha\colon {\mathfrak A}\to {\mathfrak A}'$ whose germs are $\alpha(x)$. \end{proposition}

\begin{proof} Let $a\in {\mathfrak A}(x)$ and let $\mathscr U$ be an open cover of the kind described in the statement.   Consider the family of elements $b_U =\alpha(U) (a_{|U} \in {\mathfrak A}'(U)$, $U\in\mathscr U$. If $U,V\in\mathscr U$ and $x\in U\cap V$, the previous proposition shows that $b_{U|(U\cap V)} = b_{V|(U\cap V)}$.  Thus the $\{b_U\}$ form a compatible family, so by the gluing axiom they are restrictions of $b\in {\mathfrak A}'(X)$.  We define $\alpha(a)=b$.   \end{proof}

\begin{corollary}\label{extcor} Let $\mathscr B$ be a basis for the topology of $X$ and let $\mathfrak A$ and ${\mathfrak A}'$ be sheaves of $C^*$-algebras over $X$.  Any natural transformation between the restrictions of $\mathfrak A$ and ${\mathfrak A}'$ to the full subcategory on the members of $\mathscr B$ extends uniquely to a morphism of sheaves. \quad\qedsymbol\end{corollary}

\begin{remark}\label{secremark}Since a section of a sheaf is an example of a morphism (a section of $\mathfrak A$ is a morphism of the constant sheaf to $\mathfrak A$), a special case of these results is the following local characterization of sections: if $\mathfrak A$ is a sheaf of $C^*$-algebras and $a(x)\in{\mathfrak A}(x)$ for all $x$, and if each $x\in X$ has a neighborhood $U$ for which there is $a\in {\mathfrak A}(U)$ restricting to $a(y)$ for all $y\in U$, then the $a(x)$ arise from a section $a$ of $\mathfrak A$. \end{remark}

The following localization property for sheaves of $C^*$-algebras is the relative version of proposition~\ref{locsec}.

\begin{proposition}\label{reliso} Let $\alpha\colon {\mathfrak A}\to {\mathfrak A}'$ be a morphism of sheaves of $C^*$-algebras over $X$, and suppose that $\mathfrak A$ is special.  If the germ $\alpha(x)\colon {\mathfrak A}(x)\to {\mathfrak A}'(x)$ is an isomorphism for each $x\in X$, then $\alpha$ is an isomorphism of sheaves (and, in particular, $\alpha(X)\colon {\mathfrak A}(X)\to {\mathfrak A}'(X)$ is an isomorphism). \quad\qedsymbol\end{proposition}

\begin{proof} It follows directly from Proposition~\ref{locsec} that if $\alpha(x)$ is injective for each $x$, then $\alpha(X)$ is injective.  Now suppose further that $\alpha(x)$ is an \emph{isomorphism} for each $x$, and let $b\in {\mathfrak A}'(X)$.  Then, given any $\epsilon>0$, there exists an open cover $\mathscr U$ of $X$ and, for each $U\in\mathscr U$ an element $a_U\in {\mathfrak A}(U)$ such that $\|\alpha(U)(a_U) - b_{|U}\| < \epsilon/2$. Moreover if $U,V\in\mathscr U$ we have
\[ \|\alpha(U\cap V) (a_{U|V} - a_{V|U}) \| < \epsilon, \]
and therefore (since an injective $*$-homomorphism is isometric), $\|a_{U|V} - a_{V|U}\| < \epsilon$.   Thus we are in a position to apply Lemma~\ref{puglue} and obtain $a\in{\mathscr A}(X)$ with $\|\alpha(a)-b\|<2\epsilon$.  Since the range of $\alpha$ is closed, this suffices to show that $\alpha$ is surjective.  \end{proof}

\subsection{Appendix}  The literature contains a number of (apparently) different notions of a ``continuous family'' of $C^*$-algebras parameterized by $X$.  We briefly review these ideas here.  References for this subsection are~\cite{ara_sheaves_2010}, \cite{hofmann_bundles_1977} and~\cite[Appendix C]{williams_crossed_2007}.

\begin{definition} Let $A$ be a $C^*$-algebra.  One says that $A$ is a \emph{$C_0(X)$-algebra} if there is given a $*$-homomorphism $C_0(X) \to {\mathscr Z}(M(A))$, the center of the multiplier algebra of $A$.  In particular, then, $A$ becomes a $C_0(X)$-module, and we will  require that this structure is \emph{essential} in the sense that  $C_0(X)A=A$.   \end{definition}

Let $A$ be a $C_0(X)$-algebra, and for any closed subset $K$ of $X$ let $I_K$ denote the ideal of functions in $C_0(X)$  that vanish on $K$.  Then $I_KA$ is a closed ideal in $A$.  In particular we can consider the quotient algebras
\[ A_x = A / I_{\{x\}}A , \quad x\in X \]
which  form a family of $C^*$-algebras parameterized by $X$.   In fact they form an \emph{upper semicontinuous $C^*$-bundle} $\mathscr A$ in the sense of \cite[Definition 5.1]{ara_sheaves_2010}.  Conversely, given such an usc $C^*$-bundle $E$ over $X$, one can form its algebra of continuous sections vanishing at infinity, and this is a $C_0(X)$-algebra.
These constructions   establish a one-to-one correspondence between
essential $C_0(X)$-algebras and upper semicontinuous $C^*$-bundles over $X$,
in the case of a locally compact metrizable base space $X$.

Given an upper semicontinuous $C^*$-bundle $\mathscr A$, one can consider the functor which assigns to each open subset $U$ of $X$ the space ${\mathscr A}_b(U)$ of bounded continuous sections of $\mathscr A$ over $U$.  It is easy to verify that this functor is a \emph{sheaf} of unital $C^*$-algebras.  Moreover, this is a \emph{special} sheaf: ${\mathscr A}_b(U)$ is a module over $C_b(U)$ for every open set $U$.  Conversely, given a special sheaf of unital $C^*$-algebras, it can be shown that its stalks in fact form an usc $C^*$-bundle.   Thus every special sheaf arises from this construction.

As mentioned above, we will not use the results of this subsection in the main part of the paper.  Finally, note that the equivalences stated above depend on some assumptions on the topology of $X$ --- our hypothesis that $X$ is a proper metric space is stronger than necessary, but Hausdorff and paracompact at least are needed.   In several classical applications of bundle theory for $C^*$-algebras (though not in ours) these hypotheses do not apply.

\section{The sheaf of noncommutative symbols}
Let $X$ be locally compact and metrizable and let a Hilbert space representation $\rho\colon C_0(X)\to\Bounded(H)$ be given. As remarked in the introduction, we will call the Hilbert space $H$, provided with this structure, an \emph{$X$-module}.  Recall that any such representation $\rho$ extends to a representation of the algebra of bounded Borel functions on $X$.  We do \emph{not} require that the representation be essential (in other words, the projection corresponding to the Borel function 1 may not be the identity).

Let ${\mathfrak C}(X)$ and ${\mathfrak D}(X)$ be the algebras of locally compact and pseudolocal operators, respectively, on the $X$-module $H$.  (We'll write ${\mathfrak C}(X;H)$ and so on if it's necessary to specify the module explicitly.)   Let ${\mathfrak Q}(X)$ denote their quotient, ${\mathfrak D}(X)/{\mathfrak C}(X)$.  We refer to~\cite{HR3} for the basic properties of these objects and in particular for the definition~\cite[5.4.3]{HR3} of $K$-homology as
\[ K_i(X) = K_{i+1}({\mathfrak Q}(X))\]
whenever the $X$-module $H$ is \emph{ample}.   We shall also need \emph{Kasparov's Lemma}~\cite[5.4.6 and 5.4.7]{HR3}: an operator $T$ is pseudolocal if and only if $\rho(f)T\rho(g)$ is compact whenever $f$ and $g$ have disjoint supports.

Note that ${\mathfrak D}(X)$ contains the image of the representation $\rho$; we therefore obtain an induced homomorphism $\sigma \colon C_0(X) \to {\mathfrak Q}(X)$.

\begin{remark} We will also consider the case where, in addition to its topology, the space $X$ carries a \emph{coarse structure} $\mathscr C$ in the sense of~\cite{JR27}.  A coarse structure defines the \emph{controlled} neighborhoods of the diagonal in $X\times X$, or equivalently the \emph{uniformly bounded} covers of $X$ --- a cover $\mathscr U$ is uniformly bounded if and only if $\bigcup \{ U\times U : U\in{\mathscr U}\}$ is controlled. We assume that the coarse structure and the topology are \emph{weakly compatible}, in the sense that $X$ must admit a uniformly bounded open cover.  (This is weaker than the condition that the coarse structure be \emph{proper}, see Definition~2.22 in~\cite{JR27}.  It does still imply, however, that all compact subsets of $X\times X$ are controlled.)

An operator $T$ on $H$ is \emph{controlled} if there is a controlled set $S\subseteq X\times X$ such that $\rho(f)T\rho(g)=0$ whenever $\Supp(f)\times\Supp(g)$ is disjoint from $S$.  In the presence of a coarse structure $\mathscr C$ we may define ${\mathfrak C}(X,{\mathscr C})$ to be the $C^*$-algebra generated by the locally compact and \emph{controlled} operators on $H$ and similarly ${\mathfrak D}(X,{\mathscr C})$ to be the $C^*$-algebra generated by the pseudolocal and controlled operators\footnote{These algebras are denoted by $C^*(X)$ and $D^*(X)$ in~\cite{HR2,HR3,HR4a} where more information about their properties may be found.} and ${\mathfrak Q}(X,{\mathscr C})$ to be their quotient.    Observe that if $\mathscr C$ is the \emph{indiscrete} coarse structure (all subsets of $X\times X$ are controlled) then these definitions correspond to our earlier ones. For this reason we may sometimes omit explicit mention of the coarse structure $\mathscr C$ from our notation. \end{remark}

Note that (for any coarse structure $\mathscr C$), ${\mathfrak D}(X,{\mathscr C})$ contains the image of the representation $\rho$; we therefore obtain an induced homomorphism $\sigma \colon C_0(X) \to {\mathfrak Q}(X,{\mathscr C})$.

\begin{lemma}\label{l1} Whether or not the representation $\rho$ is essential, the homomorphism $\sigma$ is.  In fact, ${\mathfrak Q}(X,{\mathscr C};H)$ is isomorphic to ${\mathfrak Q}(X,{\mathscr C};PH)$, where $P$ is the projection corresponding to the Borel function 1 on $X$. \end{lemma}

\begin{proof} Write $T\in\Bounded(H)$ as a matrix
\[ T = \left( \begin{array}{cc} T_{11} & T_{12} \\ T_{21} & T_{22} \end{array}\right) \]
with respect to the direct sum decomposition $H = PH \oplus (1-P)H$.  In order that $T$ be pesudolocal we must have $T_{11}$ arbitrary, $T_{12}$ and $T_{21}$ locally compact, and $T_{22}$ pseudolocal.   In order that $T$ be locally compact we must have $T_{11}$ arbitrary and the other three entries locally compact.  It follows that $T\mapsto T_{22}$ defines an isomorphism ${\mathfrak Q}(X, {\mathscr C};H) \to {\mathfrak Q}(X, {\mathscr C};PH)$. \end{proof}

When working with the algebras associated to coarse structures, the following lemma is important.\footnote{In~\cite{HR3} it is incorrectly claimed that this lemma is  ``immediate'' from the definitions.}

\begin{lemma}\label{badlemma} Let $X$ be a locally compact metrizable space equipped with a weakly compatible coarse structure $\mathscr C$.  Let $T\in {\mathfrak D}(X,{\mathscr C})$ and suppose in addition that $T$ is locally compact.  Then $T\in {\mathfrak C}(X,{\mathscr C})$. \end{lemma}

In other words, we have $ {\mathfrak D}(X,{\mathscr C})\cap   {\mathfrak C}(X) = {\mathfrak C}(X,{\mathscr C})$.

\begin{proof} Let us define a \emph{controlled partition of unity} on $X$ to be a locally finite partition of unity $\{\phi_j\}$, consisting of continuous functions, such that the supports $\Supp(\phi_j)$ form a uniformly bounded cover.  Since $X$ is paracompact and admits a uniformly bounded open cover, controlled partitions of unity exist.  Fix one for the duration of the proof.

Let $T_n$ be a sequence of controlled, pseudolocal operators converging (in norm) to $T$.  We know that $T$ is locally compact, but we do \emph{not} know a priori that the $\{T_n\}$ are --- that is why the lemma is not ``immediate''.  Let $\{\phi_j\}$ be a controlled partition of unity on $X$ and let $\Phi\colon\Bounded(H)\to\Bounded(H)$ be the completely positive contraction
\[ \Phi(S) = \sum_j \rho(\phi_j^{1/2}) S \rho(\phi_j^{1/2}). \]
It is easy to check that $\Phi(S)$ is a controlled operator (whatever $S$ is), and that if $S$ is pseudolocal, then $\Phi(S)-S$ is locally compact.  In particular, $\Phi(T)= (\Phi(T)-T) + T$  is locally compact and controlled, so it belongs to ${\mathfrak C}(X,{\mathscr C})$.  Moreover, $T-\Phi(T) = \lim \left( T_n- \Phi(T_n) \right)$ is a limit of locally compact, controlled operators, hence it belongs to ${\mathfrak C}(X,{\mathscr C})$.  It follows that $T\in {\mathfrak C}(X,{\mathscr C})$ as asserted. \end{proof}

Now let $X$ be as above and let $U$ be an open subset of $X$.  A compatible coarse structure $\mathscr C$ on $X$ restricts to one on $U$. Let $P_U$ be the projection on $H$ associated (via the Borel extension of $\rho$) to the characteristic function of $U$.  Note that $P_U\in{\mathfrak D}(X,{\mathscr C})$.  Let $i_U\colon C_0(U)\to C_0(X)$ denote the extension-by-zero homomorphism.  Then for each $f\in C_0(U)$, $\rho(i_U(f))$ maps $P_UH$ to itself and vanishes on the orthogonal complement $(1-P_U)H$.  Consequently, $P_UH$ has the structure of a $U$-module.

\begin{definition} When we refer to ${\mathfrak C}(U)$, ${\mathfrak D}(U)$, and ${\mathfrak Q}(U)$, these will always be defined on $P_UH$, using the $U$-module structure explained above, and the restriction of the coarse structure (if any) on $X$. \end{definition}

Suppose  that $U$ and $V$ are open subsets of $X$, with $U\subseteq V$.  Since $P_U\in {\mathfrak D}(X)$, it is easy to see that the formula
\[ r_{U,V}(T) = P_UTP_U \]
defines  contractive linear \emph{restriction maps} ${\mathfrak C}(V)\to{\mathfrak C}(U)$, ${\mathfrak D}(V)\to{\mathfrak D}(U)$ and ${\mathfrak Q}(V)\to{\mathfrak Q}(U)$, and these have the functorial property
\[ r_{U,V} \circ r_{V,W} = r_{U,W} \]
for $U\subseteq V \subseteq W$.  The restriction maps are not  homomorphisms of \emph{algebras} for $\mathfrak C$ or for $\mathfrak D$.  However, it turns out that they \emph{are}   homomorphisms for the algebras $\mathfrak Q$.

\begin{lemma}\label{pslemma}   Let $X$ be  a locally compact metrizable space, possibly equipped with a coarse structure, and let $H$ be an $X$-module.  Then the assignment
$ U \mapsto {\mathfrak Q}(U) $
(for open subsets $U$ of $X$), together with the restriction maps $r_{U,V}$, defines a presheaf of unital $C^*$-algebras over the space $X$. \end{lemma}

\begin{proof}   We need only show that the restriction maps are algebra homomorphisms (it is clear that they respect the involution). Denote $r_{U,V}(T)$ by $T_{|U}$.  We must show (using Lemma~\ref{badlemma} in the case that a coarse structure is present) that, for $S,T\in {\mathfrak D}(V)$, the difference
\[ E:= S_{|U}T_{|U} - (ST)_{|U} = P_U S (1 - P_U)T P_U \]
belongs to ${\mathfrak C}(U)$.  But for $f\in C_0(U)$,
\[ \rho(i_U(f))E = \rho(i_U(f)) S(1-P_U)TP_U \sim S \rho(i_U(f)) (1-P_U)TP_U = 0, \]
where we have used the standard notation $\sim$ for equality modulo compacts.  Similarly $E\rho(i_U(f))\sim 0$, so $E\in {\mathfrak C}(U)$, as required. \end{proof}

Now we want to show that this presheaf is in fact a sheaf.

\begin{theorem}\label{qsheaf} The functor $\mathfrak Q$ defined in Lemma~\ref{pslemma} is in fact a sheaf of $C^*$-algebras in the sense of Definition~\ref{ssdef}. \end{theorem}

\begin{proof} We verify the uniqueness axiom.  Let $U=\bigcup U_j$ be a union of open sets and let $T\in {\mathfrak D}(U)$ represent an element of ${\mathfrak Q}(U)$ whose restriction to ${\mathfrak Q}(U_j)$ is zero for each $j$.  This implies that
$ \rho(f)T \sim 0 $
whenever $f\in C_0(U_j)$ for some $j$.  Now choose a locally finite partition of unity $\phi_j$ subordinate to the cover ${\mathscr U} = \{U_i\}$ of $U$ and for $f\in C_c(U)$ write
\[ \rho(f)T = \sum_j \rho(f\phi_j)T \sim 0\]
(the sum is finite because $f$ is compactly supported). Since $C_c(U)$ is dense in $C_0(U)$, this suffices to prove that $T$ is locally compact over $U$, and thus (appealing to Lemma~\ref{badlemma} in case the coarse structure is non-trivial) that $T\in {\mathfrak C}(U)$.  It follows that $[T]=0\in {\mathfrak Q}(U)$.

We verify the gluing axiom.  Let $T_j\in {\mathfrak D}(U_j)$ be a bounded family such that, for each pair $i,j$, $T_{i|U_j} - T_{j|U_i}$ is $(U_i\cap U_j)$-locally compact.  Choose a locally finite and \emph{controlled} partition of unity $\{\phi_j\}$ subordinate to the cover $\mathscr U$. Put
\[ T = \sum_j \rho(\phi_j^{1/2}) T_j \rho(\phi_j^{1/2}), \]
the series converging in the strong operator topology with $\|T\|\le\sup \|T_j\|$.  We must show that $T\in {\mathfrak D}(U)$ and that $T_{|{U_k}} - T_k$ belongs to ${\mathfrak C}(U_k)$.

To prove the first, again it suffices to consider $f\in C_c(U)$ and note that
\[ \rho(f)T - T\rho(f) = \sum_j \left( \rho(f\phi_j^{1/2}) T_j \rho(\phi_j^{1/2}) - \rho(\phi_j^{1/2}) T_j \rho(f\phi_j^{1/2})\right) \]
is a finite sum of compact operators and hence compact.  Thus $T$ is pseudolocal, and by construction it is controlled since the partition of unity $\{\phi_j\}$ is.

To prove the second, let $g\in C_c(U_k)$.  Then
\[ \rho(g)(T-T_k) = \sum_j \rho(\phi_j^{1/2}g) (T_j - T_k) \rho(\phi_j^{1/2}). \]
Again the sum is finite, and each $\rho(\phi_j^{1/2}g)(T_j - T_k)$ is compact by the gluing condition.   Thus $T_{|{U_k}} - T_k$ is locally compact over $U_k$.  Another appeal to Lemma~\ref{badlemma} (applied to the space $U_k$) now completes the proof.
\end{proof}

\begin{remark} Finally note that $\mathfrak Q$ is a \emph{special} sheaf.  To provide the module structure, it suffices to show that $\sigma(C_b(U))$ is \emph{central} in ${\mathfrak Q }(U)$; that is to say, $\rho(C_b(U))$ commutes modulo locally compact operators with every $T\in{\mathfrak D}(U)$.  But this is easy: if $f\in C_b(U)$, $g\in C_0(U)$ then
\[ g [T,f] = gTf - gfT \sim Tgf - gfT \sim 0 \]
as required. \end{remark}

\begin{definition} We will call $\mathfrak Q$ the \emph{sheaf of noncommutative symbols} over $X$. \end{definition}

We emphasize that these constructions are valid for any choice of (weakly compatible) coarse structure on $X$.  It is an important observation (used, for example, in the construction of assembly maps) that the resultant sheaves $\mathfrak Q$ are in fact \emph{independent} of the coarse structure.  Sheaf theory permits a very concise formulation of the proof:

\begin{proposition}\label{forgetcoarse} Let $X$ be a locally compact metrizable space, and let $\mathscr C$ be a weakly compatible coarse structure on $X$. Let $H$ be an $X$-module (assumed in the notation below).  Then the map
\[ {\mathfrak Q}(X,{\mathscr C}) \to {\mathfrak Q}(X), \]
defined by forgetting the coarse structure, is an isomorphism of algebras.  In fact, it comes from an isomorphism of the underlying sheaves. \end{proposition}

\begin{proof} Since $X$ is locally compact, every $x\in X$ has a neighborhood $U$ with compact closure.  The induced coarse structure on $U$ is then the indiscrete structure, and so
$ {\mathfrak Q}(U,{\mathscr C}) \to {\mathfrak Q}(U) $
is trivially an isomorphism. The result now follows from Proposition~\ref{reliso}. \end{proof}

This result has several applications: we give two.

\subsection{Construction of assembly maps}  Consider a space $X$ with a compatible coarse structure $\mathscr C$, as above.  By Proposition~\ref{forgetcoarse}, there is a short exact sequence of $C^*$-algebras
\[ 0 \to {\mathfrak C}(X;{\mathscr C}) \to {\mathfrak D}(X;{\mathscr C}) \to {\mathfrak Q}(X) \to 0. \]
The boundary map in the $K$-theory long exact sequence associated to this short exact sequence of $C^*$-algebras is the \emph{coarse assembly map for $X$}
\[ A_{\mathscr C} \colon K_i(X) = K_{i+1}({\mathfrak Q}(X)) \to K_i({\mathfrak C}(X;{\mathscr C})). \]
This construction of the coarse assembly map was first outlined in~\cite{HR2}.

There is a similar construction of the classical (``Baum-Connes'') assembly map (in the torsion-free case). The data here are  a \emph{compact} metric space $X$ and a normal covering space $\pi\colon \tilde{X}\to X$ with covering group $\Gamma$ (usually one considers the universal cover, but that does not make any difference here).  Let $H$ be an $X$-module.   There is an induced $(\tilde{X},\Gamma)$-module $\tilde{H}$, that is an $\tilde{X}$-module with a compatible unitary action of $\Gamma$.  This is most briefly described as
\[ \tilde{H} = E \otimes_{C(X)} H \]
where $E$ is the Hilbert $C(X)$-module of continuous sections of the ``Mischenko bundle'', the flat bundle over $X$ associated to the natural representation of $\pi_1(X)$ on $\ell^2(\Gamma)$ by deck transformations. (If $H=L^2(X,\mu)$, then $\tilde{H} = L^2(\tilde{X},\tilde{\mu})$ where $\tilde{\mu}$ is the pull-back of the measure $\mu$. ) We consider the algebras ${\mathfrak C}_\Gamma(X)$ and ${\mathfrak D}_\Gamma(X)$ and their quotient ${\mathfrak Q}_\Gamma(X)$, compare~\cite{JR26}.  Here ${\mathfrak C}_\Gamma$ denotes the norm closure of the controlled\footnote{With respect to the canonical $\Gamma$-invariant coarse structure on $\tilde{X}$.}, locally compact, $\Gamma$-invariant operators on $\tilde{H}$, and ${\mathfrak D}_\Gamma(X)$ is the same thing with ``locally compact'' replaced by ``pseudolocal''.  Again, for any open subset $U$ of $X$ we may define ${\mathfrak C}_\Gamma(U)$, ${\mathfrak D}_\Gamma(U)$ and ${\mathfrak Q}_\Gamma(U)$ starting from the $U$-module $P_UH$, and we have

\begin{proposition} For any normal $\Gamma$-covering on $X$, the construction ${\mathfrak Q}_\Gamma$ defines a (special) sheaf of unital ${\mathfrak C}$-algebras over $X$. \end{proposition}

\begin{proof} We use the same techniques as before, being careful to employ $\Gamma$-invariant partitions of unity.  \end{proof}

\begin{lemma} With the notation above, we have a (natural) isomorphism ${\mathfrak Q}_\Gamma(X) \to {\mathfrak Q}(X)$.
\end{lemma}

\begin{proof} Suppose that $U\subseteq X$ is an open set sufficiently small that $\pi^{-1}(U) \cong U\times\Gamma$ (note that there is a basis $\mathscr B$ for the topology of $X$ consisting of such sets).  Then $P_U\tilde{H} \cong P_UH \otimes \ell^2(\Gamma)$, so that any bounded operator $T$ on $P_U\tilde{H}$ can be represented by a matrix $T_{\gamma\delta}$ of bounded operators on $P_U H$.  Suppose that $T$ is pseudolocal, $\Gamma$-invariant and of finite propagation.  Then $T_{\gamma\delta}$ depends only on $\gamma\delta^{-1}$ and is zero except for finitely many values of $\gamma\delta^{-1}$, it is locally compact unless $\gamma=\delta$, and $T_{\gamma\gamma}$ is pseudolocal (and independent of $\gamma$).  It follows that the assignment
\[ T \mapsto T_{\gamma\gamma} \]
is an isomorphism ${\mathfrak Q}_\Gamma(U) \to {\mathfrak Q}(U)$.  Since these isomorphisms are defined for all $U$ belonging to the basis $\mathscr B$ and are obviously compatible, Corollary~\ref{extcor} and Proposition~\ref{reliso} show that they come from an isomorphism of sheaves.
\end{proof}

It can be shown  (see~\cite{JR26} again) that the algebra ${\mathfrak C}_\Gamma(X)$ is Morita equivalent to the reduced $C^*$-algebra $C^*_r(\Gamma)$. Thus from the above lemma and the boundary map in $K$-theory we obtain a homomorphism
\[ A_\Gamma: K_i(X) = K_{i+1}({\mathfrak Q}(X)) \to K_i(C^*_r(\Gamma)) \]
which is the Baum-Connes assembly map in this case.

\subsection{Flasqueness}   Recall that a sheaf $\mathfrak A$ is said to be \emph{flasque} (sometimes translated as \emph{flabby}) if, whenever $U\subseteq V$ are open sets, the restriction map ${\mathfrak A}(V)\to{\mathfrak A}(U)$ is an epimorphism.

\begin{proposition} Let $X$ be a locally compact metrizable space and let $H$ be an $X$-module.  Then the sheaf ${\mathfrak Q}(\cdot; H)$, defined over $X$, is flasque. \end{proposition}

Notice that, in accordance with Proposition~\ref{forgetcoarse}, we make no mention of any coarse structure on $X$.

\begin{proof} (Compare the ``commutative proof'' of Theorem 5.4.5 in~\cite{HR3}.)  It is enough to show that, for any open subset $U$ of $X$, the restriction ${\mathfrak Q}(X) \to {\mathfrak Q}(U)$ is surjective.  Moreover, in doing this, there is no loss of generality in assuming that $X$ is compact.  In fact, if $Y$ denotes the 1-point compactification of $X$, then ${\mathfrak D}(Y;H)$ is a \emph{subalgebra} of ${\mathfrak D}(X;H)$ and so it certainly suffices to prove that ${\mathfrak Q}(Y)$ surjects onto ${\mathfrak Q}(U)$.

Let us assume $X$ is compact, then, and give $U$ the \emph{topological coarse structure} $\mathscr C$ \cite[Definition 2.28]{JR27} associated to its compactification $\bar{U}\subseteq X$.  According to Proposition~\ref{forgetcoarse}, we may write
\[ {\mathfrak Q}(U) = {\mathfrak D}(U,{\mathscr C})/{\mathfrak C}(U,{\mathscr C}). \]
But now let $P$ be the projection operator corresponding to $U$, and suppose that $T$ is a $\mathscr C$-controlled and $U$-pseudolocal operator on $PH$. Suppose that $T$ is supported in a controlled subset $S$ of $U\times U$.  Let $f,g\in C(X)$ have disjoint supports and let $T'$ be the extension of $T$ by zero to an operator on $H$.  By definition of the continuously controlled coarse structure,
\[ S \cap (\Supp(f)\times\Supp(g)) \]
is a relatively compact subset of $U\times U$. It follows from the pseudolocality of $T$ that $\rho(f)T'\rho(g)$ is a compact operator.  By Kasparov's Lemma, $T'$ is pseudolocal on $X$, i.e., it belongs to ${\mathfrak D}(X)$, and it clearly maps to $T$ under restriction. This completes the proof.
\end{proof}

\begin{remark} \label{ideff} Let $Z\subseteq X$ be closed, and consider the restriction map
\[ {\mathfrak D}(X) \to {\mathfrak Q}(X\setminus Z) \]
which we have just shown to be surjective.  The kernel of this map is the ideal consisting of those pseudolocal operators $T$ on $X$ which are locally compact away from $Z$ --- that is, if $f$ vanishes on $Z$, then $\rho(f)T$ and $T\rho(f)$ are compact.   This  ideal is denoted ${\mathfrak D}_X(Z)$.
Thus (by definition) we have an isomorphism
\[ {\mathfrak D}(X)/{\mathfrak D}_X(Z) \cong {\mathfrak Q}(X\setminus Z). \]
\end{remark}

\begin{remark}\label{excis} The excision theorem in $K$-homology says that ``from the point of view of $K$-theory'' the ideal ${\mathfrak D}_X(Z)$ behaves just like ${\mathfrak D}(Z)$.  In particular, the $K$-theory of ${\mathfrak D}_X(Z)/{\mathfrak C}(X)$ is the same as that of ${\mathfrak Q}(Z)$; that is, the $K$-homology of $Z$ (with a dimension shift).  See~\cite[Chapter 5]{HR3}.  \end{remark}

In the context of the preceding remark, suppose that $X$ also has a coarse structure $\mathscr C$.  There is a version (\cite{Siegel}) of  ${\mathfrak D}_X(Z)$ that takes the coarse structure into account: namely, the algebra ${\mathfrak D}_X(Z,{\mathscr C})$ generated by the controlled, pseudolocal operators $T$ that are locally compact away from $Z$ and are also \emph{supported close to $T$}.  This latter condition means that there is a controlled set $S$ such that $\rho(f)T = T\rho(f)=0$ if $\Supp(f)\times Z $ does not meet $S$.  Similarly we may define ${\mathfrak C}_X(Z,{\mathscr C})$ to be the algebra generated by the controlled, locally compact operators that are supported close to $Z$; it is an ideal in ${\mathfrak D}_X(Z,{\mathscr C})$.  We have the following ``relative'' version of Proposition~\ref{forgetcoarse}:

\begin{proposition} In the situation above, the quotient
\[ {\mathfrak Q}_X(Z,{\mathscr C}) = {\mathfrak D}_X(Z,{\mathscr C})/{\mathfrak C}_X(Z,{\mathscr C})
 \]
 is independent of the choice of compatible coarse structure $\mathscr C$ on $X$. \end{proposition}

\begin{proof} (Outline) We proceed as for the absolute version of the same result, Proposition~\ref{forgetcoarse}.  The first step is to show that ${\mathfrak Q}_\bullet(Z,{\mathscr C})$ is a sheaf over $X$, for any coarse structure $\mathscr C$. (Of course, the stalks over all $x\in X\setminus Z$ are in fact zero.)  Next, we argue exactly as in the absolute proof that the forgetful map
\[ {\mathfrak Q}_\bullet(Z,{\mathscr C}) \to  {\mathfrak Q}_\bullet(Z) \] is an isomorphism on stalks and therefore a global isomorphism.

The proof that ${\mathfrak Q}_\bullet(Z,{\mathscr C})$ is a sheaf proceeds as in Theorem~\ref{qsheaf}.  One  modifications is needed in the relative case.  In the proof of Theorem~\ref{qsheaf}, and again in the proof of lemma~\ref{badlemma} on which it depends, use is made of a \emph{controlled partition of unity} $\{\phi_j\}$ and of certain sums like $\sum \rho(\phi_j^{1/2}) T_j \rho(\phi_j^{1/2})$ depending on it.  The modification that is needed is to take these sums \emph{only} over the index set
\[ J' = \{ j : \Supp(\phi_j)\cap Z \neq \emptyset \} .\]
The reader may verify that, with this modification, the proofs of the relative version of Lemma~\ref{badlemma} (which states that ${\mathfrak D}_X(Z,{\mathscr C})\cap {\mathfrak C}(X) = {\mathfrak C}_X(Z,{\mathscr C})$) and of the relative version of Theorem~\ref{qsheaf} (which states that ${\mathfrak Q}_\bullet(Z,{\mathscr C})$ is a sheaf over $X$) both go through in the same way as the proofs of their absolute counterparts.
\end{proof}

\begin{remark} We will discuss the homological implications of this flasqueness in a subsequent paper. \end{remark}

\section{The homology class of an elliptic operator}

In Chapter X of~\cite{HR3}, a procedure is described for associating a K-homology class to any elliptic operator on a manifold $M$ (whether or not the manifold is complete for the operator).  This process involves a number of ``partition of unity'' constructions which are conveniently formulated in the language of the sheaf of symbols ${\mathfrak Q}$.

Let $M$ be a manifold and let $D$ be a (symmetric, first order) elliptic differential operator on (the sections of some vector bundle $S$ over)  $M$.  We are going to associate a $K$-homology class to $D$.  For this purpose, recall that a \emph{normalizing function} is an odd, smooth function $\chi\colon\R\to [-1,1]$ such that $\chi(\lambda)\to \pm 1$ as $\lambda\to\pm\infty$ (see~\cite[Definition 10.6.1]{HR3}).

 Let $H$ denote the $X$-module $L^2(S)$. Fix a normalizing function $\chi$.  For each $x\in M$, choose an open set $U$ containing $x$ and an essentially selfadjoint differential operator $D'$ that agrees with $D$ on $U$.  By~\cite[Lemma 10.6.4]{HR3}, $\chi(D')$ commutes modulo compact operators with every function $g\in C_0(U)$.  In other words, it defines an element of ${\mathfrak D}(U)$.

\begin{definition} The \emph{noncommutative symbol} of $D$ at $x$, denoted $\sigma_D(x)$, is the equivalence class defined by $\chi(D')$ in the stalk
\[ {\mathfrak Q}(x) = \lim_{V\ni x} {\mathfrak Q}(V) = \lim_{V\ni x} {\mathfrak D}(V)/{\mathfrak C}(V). \]
\end{definition}

In order that $\sigma_D(x)$ be well defined, we need the following proposition:

\begin{lemma} The element $\sigma_D(x)$ (of the stalk of $\mathfrak Q$ at $X$) defined above is a symmetry (a selfadjoint involution), independent of all the choices involved in its construction, namely those of the neighborhood $U$, the essentially self-adjoint extension $D'$, and the normalizing function $\chi$. Moreover, as $x$ varies, the $\sigma_D(x)$ form a   section of the sheaf $\mathfrak Q$. \end{lemma}

\begin{proof} Two normalizing functions differ by some $\phi\in C_0(\R)$, and Proposition 10.4.1 of~\cite{HR3} shows that the corresponding operators differ by $\phi(D') \in {\mathfrak C}(U)$.   Moreover, since $\chi$ is real-valued  and $\chi^2-1\in C_0(\R)$ for any normalizing function $\chi$, the equivalence class of $\chi(D')$ must indeed be a self-adjoint involution.

On the other hand, Lemma 10.8.4 of~\cite{HR3} shows that if two essentially selfadjoint operators $D'$ and $D''$ agree on some neighborhood $V$ of $x$ (possibly smaller than $U$), then there is \emph{some} normalizing function $\chi$ such that $\chi(D')\rho(g)=\chi(D'')\rho(g)$ for all $g\in C_0(V)$.

The final statement in the proposition (that we indeed have defined a section of the sheaf $\mathfrak Q$) follows from Remark~\ref{secremark}.
\end{proof}

We have not mentioned yet the possibility that the underlying vector bundle $S$, and therefore the Hilbert space $H$, is \emph{graded} (or multigraded). Such a grading carries through the whole discussion, and it indexes the kind of $K$-theory class determined by the operator in question.  An ungraded symmetry in a $C^*$-algebra determines a $K_0$ class; an odd, grdaed symmetry determines a $K_1$ class; and, if we are using multigradings, a $p$-multigraded symmetry determines a $K_p$ class.  See~\cite{HR3,vD1}.

\begin{definition} Let $D$ be an elliptic operator as above.  The $K$-theory class  of the (possibly graded) symmetry $\sigma_D$ in $K_{i+1}({\mathfrak Q}(X))$ is by definition the \emph{$K$-homology class} $[D]$ of the operator $D$, in the group $K_i(X)$. \end{definition}

Note that this agrees with Definition~10.8.3 of~\cite{HR3}.  We have formulated our present discussion in terms of symmetric, first order differential operators in order to connect directly with the exposition in~\cite{HR3}.  However, note that we could equally well express matters in terms of pseudodifferential operators (of any order).   Indeed, the symbol calculus for pseudodifferential operators~\cite{Tay} directly gives a homomorphism from the algebra of principal symbols at $x$ to the stalk ${\mathfrak Q}(x)$.

Analytic $K$-homology is a ``locally finite'' homology theory in the language of topology, and in particular there exist natural restriction maps $j_U\colon K_*(X)\to K_*(U)$, for any open subset $U$ of $X$.  These are just the result of applying $K$-theory to the restriction maps of the sheaf $\mathfrak Q$.  Indeed, the restriction of elliptic operators to open subsets becomes particularly straightforward from the sheaf-theoretic point of view.

\begin{proposition}\label{restrict} Let $M$ be a smooth manifold and let $D_M$ be a (symmetric first order) elliptic operator on $M$.  let $U$ be an open subset of $M$ and let $D_U$ be the restriction of $D$ to $U$.  Then $j_U[D_M]=[D_U]$, where $j_U$ is the restriction map defined above. \end{proposition}

\begin{proof} By its construction, the noncommutative symbol of an (elliptic) operator at a point $x$ depends only on the behavior of that operator in a neighborhood of $x$.  But, for $x\in U$, $D_U$ and $D_M$ agree on a neighborhood of $x$ (namely, $U$). \end{proof}

\subsection{Pairs of operators and relative homology}
We now consider the relative version of the foregoing discussion, which underlies the ``relative index theorem'' of~\cite{GrLa,JR12} and elsewhere.  The simplest example is of a manifold carrying \emph{two} elliptic operators that agree outside of some closed subset.   More generally, let $X$ be a locally compact metrizable space and $Z$ a closed subset.  We consider the following \emph{relative elliptic data} over $(Z,X)$:
\begin{enumerate}[(a)]
\item A pair of manifolds $M_1, M_2$ equipped with proper continuous \emph{control maps} $c_1,c_2$ to $X$.
\item On each manifold $M_k$ an elliptic operator $D_k$ (symmetric, first order, differential, in accordance with our standing assumptions) operating on sections of a bundles $S_k$ (possibly graded).
\item A diffeomorphism $h\colon W_1\to W_2$, where $W_k = c_k^{-1}(X\setminus Z)$, commuting with the control maps and covered by a bundle isomorphism $S_{1|W_1} \to S_{2|W_2}$ (preserving the gradings, if any) which intertwines the restrictions $D_{1|W_1} $ and $D_{2|W_2}$.
\end{enumerate}
We want to construct from this data a $K$-homology class for $Z$ that measures the ``difference'' between the homology classes of $D_1$ and $D_2$.

To carry out the construction, let $H_1$ and $H_2$ denote the Hilbert spaces $L^2(M_1,S_1)$ and $L^2(M_2,S_2)$.  Via the control maps, we can consider them as $X$-modules.  Let $Q_k = {\mathfrak Q}(X;H_k)$, for $k=1,2$, and let $J_k\lhd Q_k$ denote the ideal ${\mathfrak Q}_X(Z;H_k)$.

Write $H_k = L^2(W_k,S_k)\oplus L^2(M_k\setminus W_k, S_k)$ and let $v\colon H_1\to H_2$ be an isometry which is induced by the diffeomorphism $h$ on the first factor of  the direct sum and is zero on the second factor.  Then $\Ad(v): T\mapsto v^*Tv$ maps $\Bounded(H_1)$ to $\Bounded(H_2)$ and we have

\begin{lemma} The homomorphism $\Ad(v)$ induces an isomorphism \[\alpha\colon Q_1/J_1 \to Q_2/J_2.\] \end{lemma}

\begin{proof} It is only necessary to observe that $v^*v-1\in J_2$. \end{proof}

\begin{remark} Note, in fact, that $A_k/J_k = {\mathfrak Q}(X\setminus Z; H_k)$, via the isomorphism of Remark~\ref{ideff}. \end{remark}

Now, following~\cite{JR12}, let $A$ be the ``double''
\[ A = \{(a_1,a_2): a_k \in Q_k, \ \alpha[a_1] = [a_2] \in A_2/J_2. \]
The algebra $A$ fits into a pull-back diagram
\[  \xymatrix{ A \ar[r]\ar[d] & A_2 \ar[d]\\
A_1\ar[r] & A_1/J_1=A_2/J_2 } \]
There is a short exact sequence
\[ 0 \to J_2 \to A \to A_1 \to 0 \]
with maps $j \mapsto (0,j)$, $(a_1,a_2) \mapsto [a_1]$,
and this exact sequence is split by the diagonal map $a_1 \mapsto (a_1, \Ad(v) a_1)$.  Thus there is a canonical isomorphism
\begin{equation}\label{can}K_*(A) \cong K_*(J_2) \oplus K_*(A_1).\end{equation}
  And we note that by the excision theorem in $K$-homology, $K_{*+1}(J_2) \cong K_*(Z)$, the $K$-homology of the closed subset $Z$.

Now let $D_1$ and $D_2$ be the elliptic operators from the given set of relative data. Their noncommutative symbols $\sigma_{D_1}$ and $\sigma_{D_2}$ define symmetries in $A_1$ and $A_2$ respectively.   Moreover, the images of these symmetries in $A_k/J_k = {\mathfrak Q}(X\setminus Z; H_k)$ are the noncommutative symbols of the restrictions of $D_k$ to the $W_k$: since these operators agree here, the corresponding symmetries also agree (under the isomorphism $\alpha$).  In other words, $(\sigma_{D_1},\sigma_{D_2})$ defines a symmetry $F$ in the double algebra $A$.

\begin{definition} The component in $K_{*+1}(J_2) = K_*(Z)$ (under the isomorphism of Equation~\ref{can}) of the $K$-theory class of this symmetry $F$ is called the \emph{relative homology class} of the given set of relative elliptic data. \end{definition}

As in the absolute case (Proposition~\ref{restrict}), the construction is functorial under open inclusions $U\subseteq X$, provided that $Z\subseteq U$; the proof is the same.  Since the receiving group $K_*(Z)$ remains unchanged under such open inclusions, we obtain

\begin{proposition} The relative homology class of a set of relative elliptic data is unchanged if we restrict the data to any neighborhood of $Z$.  Consequently, this relative homology class only depends on the behavior of the data in a neighborhood of $Z$. \end{proposition}

Compare Gromov and Lawson's formulation of their relative index theorem,~\cite{GrLa}.  A more general relative index theorem in coarse geometry, based on the proposition above, can be found in~\cite{JR33}.

\end{document}